\documentclass[oneside,english]{amsart}
\usepackage[T1]{fontenc}
\usepackage[latin9]{inputenc}
\usepackage{color}
\usepackage{babel}
\usepackage{amsthm}
\usepackage{amssymb}
\usepackage{esint}
\usepackage[unicode=true,pdfusetitle,
 bookmarks=true,bookmarksnumbered=false,bookmarksopen=false,
 breaklinks=false,pdfborder={0 0 0},backref=false,colorlinks=true]
 {hyperref}

\makeatletter
\numberwithin{equation}{section}
\numberwithin{figure}{section}
  \theoremstyle{plain}
  \newtheorem*{thm*}{\protect\theoremname}
  \theoremstyle{remark}
  \newtheorem*{note*}{\protect\notename}
\theoremstyle{plain}
\newtheorem{thm}{\protect\theoremname}
  \theoremstyle{definition}
  \newtheorem{defn}[thm]{\protect\definitionname}
  \theoremstyle{remark}
  \newtheorem*{rem*}{\protect\remarkname}
  \theoremstyle{plain}
  \newtheorem{lem}[thm]{\protect\lemmaname}
  \theoremstyle{plain}
  \newtheorem{cor}[thm]{\protect\corollaryname}
  \theoremstyle{plain}
  \newtheorem{prop}[thm]{\protect\propositionname}
  \theoremstyle{plain}
  \newtheorem{question}[thm]{\protect\questionname}

\makeatother

  \providecommand{\corollaryname}{Corollary}
  \providecommand{\definitionname}{Definition}
  \providecommand{\lemmaname}{Lemma}
  \providecommand{\notename}{Note}
  \providecommand{\propositionname}{Proposition}
  \providecommand{\questionname}{Question}
  \providecommand{\remarkname}{Remark}
  \providecommand{\theoremname}{Theorem}
\providecommand{\theoremname}{Theorem}

\begin{document}

\title{A note on non-negatively curved Berwald spaces}

\author{Martin Kell}

\email{mkell@ihes.fr}

\date{\today}

\address{Institut des Hautes Études Scientifiques, 35 route de Chartres, F-91440
Bures-sur-Yvette, France}

\thanks{During the research the author was funded by an EPDI-fellowship at
the IH\'ES. He wants to thank the IH\'ES for its hospitality and
financial support. Thanks also to J. Szilasi for suggesting a more
direct approach using only the affine connection of the Berwald space.}
\begin{abstract}
In this note it is shown that Berwald spaces admitting the same norm-preserving
torsion-free affine connection have the same (weighted) Ricci curvatures.
Combing this with Szab\'o's Berwald metrization theorem one can apply
the Cheeger-Gromoll splitting theorem in order to get a full structure
theorem for Berwald spaces of non-negative Ricci curvature. Furthermore,
if none of the factor is a symmetric space one obtains an explicit
expression of Finsler norm of the resulting product. By the general
structure theorem one can apply the soul theorem to the factor in
case of non-negative flag curvature to obtain a compact totally geodesics,
totally convex submanifolds whose normal bundle is diffeomorphic to
the whole space.

In the end we given applications to the structure of Berwald-Einstein
manifolds and non-negatively curved Berwald spaces of large volume
growth. 
\end{abstract}
\maketitle
In geometry and analysis, curvature is an important tool that rules
regularity of solution of PDEs and helps to classify spaces. The two
main curvature indicators are sectional and Ricci curvature. Assuming
global bounds on those quantities one can show that a manifold must
have certain topological type. Cheeger and Gromoll managed \cite{Cheeger1971}
to show that a Riemannian manifold with non-negative Ricci curvature
contains an isometrically embedded if and only if it isometrically
splits as a Cartesian product of another manifold of non-negative
Ricci curvature and a line. Later Perelman showed \cite{Perelman1994}
that Riemannian manifold of non-negative sectional curvature diffeomorphic
to the normal bundle of an embedded compact totally geodesic submanifold.
This submanifold is called soul and every other such submanifold must
be isometric.

On Finsler manifold, i.e. those whose tangent spaces are not Euclidean
spaces, several notions of curvature exist. The most prominent being
flag and Ricci curvature. There are result on constantly curved space
and strictly positively curved spaces (see e.g. \cite{Shen1996,Shen2002}).
Only recently inspired by the theory of weighted Ricci curvature bounds
Ohta \cite{Ohta2013a} managed to show a diffeomorphic splitting theorem
for general Finsler manifolds of non-negative Ricci curvature. 

A subclass of Finsler manifolds are Berwald spaces. Whereas general
Finsler manifolds can have non-isometric tangent spaces, Berwald spaces
are modeled on the ``same'' tangent space. An important property
of those spaces is that the parallel transport is linear and admits
a unique affine connection. In \cite{Szabo1981} (see also \cite{Szabo2006})
Szab\'o managed to show that there is indeed a Riemannian metric
whose Levi-Civita connection agrees with the affine connection of
the Berwald space. In particular, geodesics in a Berwald space are
affinely equivalent to geodesics of a Riemannian manifold. As the
theory of Levi-Civita connections and their structure is well-known,
in particular their product structure (de Rham Decomposition Theorem
\cite{DeRham1952}) and their rigidity (Berger-Simmons Theorem, see
\cite{Olmos2005} for a new geometric proof), one can exactly tell
which Levi-Civita connections can be metrized via a non-Riemannian
Berwald metric. 

Because of their nicely behaved connections, there are some results
only known to hold for Berwald spaces, but not (yet) for general Finsler
spaces, e.g. Busemann convex Berwald spaces are exactly the ones having
non-positively flag curvature \cite{Kristaly2006}. Furthermore, for
Berwald spaces of non-negative Ricci curvature, Ohta \cite{Ohta2013a}
proved an extended diffeomorphic splitting theorem, i.e. he showed
that a Berwald space has a maximal flat factor and the fibers over
this factor are totally geodesic.

In this article we first show that Ricci curvature only depends on
the induced connection and thus is an affine invariant of Berwald
spaces. As Szab\'o showed that every Berwald space is affinely equivalent
to a Riemannian manifold, one can apply the classical Cheeger-Gromoll
splitting. This together with the general structure theorem of Berwald
space yields. Note that this calculation only needs the Berwald Metrization
Theorem and results from Riemannian geometry. In particular, the de
Rham Decomposition Theorem.
\begin{thm*}
A geodesically complete simply connected Berwald space $(M,F)$ of
non-negative Ricci curvature has the following form 
\[
M=M_{0}\times\cdots\times M_{n}
\]
where $M_{0}=\mathbb{R}^{k}$ is the maximal flat factor and each
$M_{i}$ is simply connected and either an affinely rigid Riemannian
manifold or a higher rank symmetric Berwald space of compact type.
In particular, each factor has non-negative Ricci curvature.

Furthermore, if none of the $M_{i}$ is a higher rank symmetric Berwald
space then the Finsler norm has the following form 
\[
F(v_{0},\cdots,v_{n})=G(v,F_{1}(v_{1}),\cdots,F_{n}(v_{n}))
\]
where $v_{i}\in TM_{i}$, $F_{i}$ is a norm induced by a unique up
to scale Riemannian metric $g_{i}$ on $M_{i}$ and $G$ is a Minkowski
norm on $\mathbb{R}^{k+n}$ which is symmetric in the last $n$-coordinates.
\end{thm*}
Let $g_{i}$ denote either the unique Riemannian metric on $M_{i}$
or a Riemannian metric on the symmetric space $M_{i}$ such that $(M_{i},g_{i})$
is Einstein, e.g. $g_{i}=\operatorname{Ric}_{i}$. Note that in both
cases $(M_{i},F_{i})$ and $(M_{i},g_{i})$ are affinely equivalent.
Then the Cartesian product $(M,g)$ of $(M_{i},g_{i})$ with Riemannian
metric $g=\sum g_{i}$, where $g_{0}$ is any flat metric on $M_{0}$,
has non-negative sectional curvature if $M$ has non-negative flag
curvature. Furthermore, $(M,F)$ and $(M,g)$ are affinely equivalent,
i.e. their geodesics agree. If one applies the soul theorem \cite{Perelman1994}
to $(M,g)$ then the soul $S$ is also a compact, totally geodesic
submanifold for the Berwald space $(M,F)$. In particular, $M$ is
diffeomorphic to the normal bundle of $S$.
\begin{note*}
An earlier version of this paper relied heavily on \cite{Ohta2013a}
and its use of the Chern connection. The author wants to thank Szilasi
for referring him to the paper \cite{Deng2015} and suggesting to
avoid complicated calculation involving the Chern connection.
\end{note*}

\section{Affine connections and the de Rham decomposition}

Let $M$ be a connected, $n$-dimensional smooth manifold and denote
by $C^{\infty}(M,TM)$ be the space of vector fields. 
\begin{defn}
[Affine connection] An \emph{affine connection} is a bilinear map
\begin{eqnarray*}
\nabla:C^{\infty}(M,TM)\times C^{\infty}(M,TM) & \to & C^{\infty}(M,TM)\\
(X,Y) & \mapsto & \nabla_{X}Y
\end{eqnarray*}
such that for all smooth function $f$ and all vector field $X,Y$
on $M$
\begin{itemize}
\item \emph{($C^{\infty}(M,\mathbb{R})$-linearity)} $\nabla_{fX}Y=f\nabla_{X}Y$
\item \emph{(Leibniz rule)} $\nabla_{X}(fX)=df(X)Y+f\nabla_{X}Y$.
\end{itemize}
We say that $\nabla$ is \emph{torsion-free} if for all vector field
$X$ and $Y$ 
\[
\nabla_{X}Y-\nabla_{Y}X=[X,Y].
\]
\end{defn}
\begin{rem*}
Note that every connection has an associated covariant derivative
which agrees with the connection when applied to vector fields. In
the Finsler setting the gradient of functions, usually denoted by
$\nabla$ as well, is defined via Legendre transform and does not
agree with the connection or the covariant derivative applied to $f$.
Therefore, it is more convenient to use $D$ as a symbol for the covariant
derivative. 
\end{rem*}
Suppose the tangent bundle $TM$ splits into a direct sum $V\oplus W$
such $\nabla_{X}(Y+Z)=\nabla_{X}Y$ whenever $X,Y\in V$ and $Z\in W$
or $X,Y\in W$ and $Z\in V$. Then $\nabla$ can be written as a sum
of two affine connections only acting on $V$ and $W$ respectively.
In this case one says that $\nabla$ is reducible. If there does not
exist such a splitting then $\nabla$ is said to be irreducible. 

In local coordinates one can define the coefficients $\Gamma_{ij}^{k}$
of the connection as follows
\[
\nabla_{\partial_{i}}\partial_{j}=\sum_{k=1}^{n}\Gamma_{ij}^{k}\partial_{k}
\]
where $\{\partial_{i}\}$ spans is a local trivialization of the tangent
bundle. Then the covariant derivative $D$ acting on vector fields
is defined as 
\[
D_{V}X=\sum_{i,j=1}^{n}\left\{ v^{i}\partial_{j}X^{i}+\sum_{k=1}^{n}\Gamma_{jk}^{i}v^{j}X^{k}\right\} \partial_{i}
\]
where $X=\sum_{i=1}^{n}X^{i}\partial_{i}$ and $V=\sum_{i=1}^{n}v^{i}\partial_{i}$.

If $(M,g)$ is a Riemannian manifold then there is a unique torsion-free
affine connection, called \emph{Levi-Civita connection}, which is
also metric compatible, i.e. 
\[
\nabla g=0.
\]
This is equivalent to require the connection coefficients to have
the following form
\[
\Gamma_{ij}^{k}=\frac{1}{2}\sum_{m=1}^{n}g^{km}\left(\partial_{j}g_{mi}+\partial_{i}g_{mj}-\partial_{m}g_{ij}\right)
\]
where $g=(g_{ij})$ and $(g^{ij})$ is the inverse of $(g_{ij})$.
In the Riemannian case the coefficients are usually called Christoffel
symbols.

For a simply connected Riemannian manifold which is geodesically complete
(see section on Finsler structures below) one can split the manifold
into a product of irreducible components, i.e. those having irreducible
Levi-Civita connection. The following theorem is well-known.
\begin{thm}
[de Rham Decomposition] Let $M$ be a simply connected, geodesically
complete Riemannian manifold with Levi-Civita connection $\nabla$.
Then $M$ is isometric to a product, i.e. 
\[
M=M_{0}\mbox{\ensuremath{\times}}\cdots\times M_{n}
\]
metric $g=\sum g_{i}$ where $(M_{0},g_{0})$ is isometric to an Euclidean
space and each $(M_{i},g_{i})$ admits a unique irreducible Levi-Civita
connection $\nabla_{i}$. Furthermore, the Levi-Civita connection
$\nabla$ can be written as a sum of the $\nabla_{i}$. 
\end{thm}
Every affine connection induces a unique linear \emph{parallel transport}
along curves, i.e. for a smooth curve there is a map
\[
P_{\gamma}:T_{\gamma_{0}}M\to T_{\gamma_{1}}M
\]
such that $P_{\gamma}v$ is the unique vector $X_{\gamma_{1}}$ of
a vector field $X$ such that $\nabla_{\gamma_{t}}X=0$ for all $t\in[0,1]$.
It is easy to see that $P_{\gamma}$ is an invertible linear map between
vector spaces. We define the holonomy group as those invertible linear
maps $A:T_{x}M\to T_{x}M$ for which there is a curve $\gamma$ with
$\gamma(0)=\gamma(1)=x$ and $A=P_{\gamma}$. Note that if $\nabla$
is a Levi-Civita connection then the parallel transport preserves
the norm of a vector. In particular, every $A\in\mathcal{H}_{x}$
is an orthogonal transformation on $(T_{x}M,g_{x})$.
\begin{thm}
[Berger-Simmons] Let $(M,g)$ be a simply connected, geodesically
complete, irreducible Riemannian manifold. If the holonomy does not
act transitively on the unit sphere then $(M,g)$ is a Riemannian
symmetric space. 
\end{thm}
It is not important to know the exact definition of symmetric spaces.
We only need the following facts: every symmetric space has a unique
rank which is a natural number $\ge1$. And every higher rank symmetric
space is either non-compact or compact and embeds a totally geodesic
flat submanifold of dimension equal to its rank. In case the symmetric
space is non-compact and irreducible this flat submanifold is isometric
to $\mathbb{R}^{n}$, in particular, there exists an isometrically
embedded line.
\begin{lem}
[Holonomy orbits in rank 1] \label{lem:holonRk1}Assume $(M,g)$
is a symmetric space of rank $1$. Denote by  $\mathcal{H}_{x}$ the
holonomy group of the Levi-Civita connection acting on $T_{x}M$ then
for $\mathcal{H}_{x}v\cap V\ne\varnothing$ for every $v\in T_{x}M$
and and one dimensional subspace $V$ of $T_{x}M$. In words, the
orbit of the holonomy intersects every one dimensional subspace.\end{lem}
\begin{proof}
This is a well-known fact from the theory of symmetric spaces. Just
note that every one dimensional subspace of $T_{x}M$ is a Cartan
subalgebra and every orbit of the holonomy group intersects this algebra,
see \cite[Lemma 3.1]{Szabo2006}. 
\end{proof}

\section{Finsler structures and the Chern connection}

Let $M$ be a connected, $n$-dimensional $C^{\infty}$-manifold. 
\begin{defn}
[Finsler structure] A $C^{\infty}$-Finsler structure on $M$ is
a function $F:TM\to[0,\infty)$ such that the following holds
\begin{enumerate}
\item (Regularity) $F$ is $C^{\infty}$ on $TM\backslash\{\mathbf{0}\}$
where $\mathbf{0}$ stands for the zero section,
\item (Positive homogeneity) for any $v\in TM$ and any $\lambda>0$, it
holds $F(\lambda v)=\lambda F(v)$,
\item (Strong convexity) In local coordinates $(x^{i})_{i=1}^{n}$ on $U\subset M$
the matrix
\[
\left(g_{ij}(v)\right):=\left(\frac{1}{2}\frac{\partial^{2}(F^{2})}{\partial v^{i}\partial v^{j}}(v)\right)
\]
is positive-definite at every $v\in\pi^{-1}(U)\backslash0$ where
$\pi:TM\to M$ is the natural projection of the tangent bundle.
\end{enumerate}
\end{defn}
Strictly speaking, this is nothing more than defining a Minkowski
norm $F|_{T_{x}M}$ on each $T_{x}M$ with some regularity requirements
depending on $x$. We don't require $F$ to be absolutely homogeneous,
i.e. $F(v)\ne F(-v)$ is possible. In such a case the ``induced''
distance is not symmetric. 

If each $F_{x}$ is induced by an inner product $g_{x}$ then the
Finsler structure is actually a Riemannian structure and $(M,g)$
a Riemannian manifold. Note that in that case $(g_{ij}(v_{x}))=g_{x}$
and the distance is symmetric.

A geodesic from $x$ to $y$ is a curve $\gamma:[0,1]\to M$ which
minimizes the following functional 
\[
\gamma\mapsto\int_{0}^{1}F(\dot{\gamma}_{t})dt.
\]
Note that in general that the reversed curve $\bar{\gamma}(t)=\gamma(1-t)$
is not a geodesic from $x$ to $y$. Then there is a (possibly asymmetric)
metric $d$ defined as 
\[
d(x,y)=\inf_{\gamma}\int_{0}^{1}F(\dot{\gamma}_{t})dt
\]
where the infimum is taken over all curves $\gamma:[0,1]\to M$ connecting
$x$ and $y$.

\section{Berwald spaces and affine rigidity}

In contrast to the Riemannian setting there is no unique (affine)
connection on a general Finsler manifolds. However, there is a subclass
containing the Riemannian manifolds where this is the case.
\begin{defn}
[Berwald space] A Finsler structure $F$ on $M$ is called Berwald
if it admits a (unique) torsion-free affine connection whose induced
parallel transport preserves the Finsler norm. \end{defn}
\begin{rem*}
In this note we avoid the use of the Chern connection. Usually one
requires the Chern connection to be an affine connection, i.e. independent
of the reference vector. The opposite statement is shown in \cite{Szilasi2011},
i.e. a torsion-free, norm preserving affine connection is the Chern
connection and the space is Berwald.
\end{rem*}
Note that for general functions $G:TM\to\mathbb{R}$ one can interpret
$\nabla G=0$ as being invariant under parallel transport induced
by $\nabla$. Hence for Berwald spaces might say that $\nabla F=0$
is metric compatibility similar to the Riemannian case.

It is well-known that every geodesic satisfies the following 
\[
D_{\dot{\gamma}}\dot{\gamma}=0
\]
and $F(\dot{\gamma})\equiv const$, i.e. it is a constant speed auto-parallel
curve. Also the converse holds; every auto-parallel curve is locally
a geodesic, i.e. it is locally the distance minimizing curve between
two points. We say that a $(M,F)$ is forward geodesically complete
if the space is complete and any auto-parallel curve $\gamma:[0,t]\to M$
with $D_{\dot{\gamma}}^{\dot{\gamma}}\dot{\gamma}=0$ can be extended
to an auto-parallel curve beyond $t$. 

In particular, if a connection preserves two different Finsler norms
$F_{1}$ and $F_{2}$ on $M$ then their auto-parallel curves agree
and their geodesics are the same, i.e. $\gamma$ is auto-parallel
w.r.t. the first spaces iff it is w.r.t. the second space; geodesics
only differ in their speed. In such a case we say that the two spaces
are \emph{affinely equivalen}t.
\begin{lem}
Let $F_{1}$ and $F_{2}$ be two Berwald structure on $M$ such that
their induced connections agree. Then they are affinely equivalent.
\end{lem}
Because on a Berwald space $F$ and the reverses Finsler structure
$\bar{F}(v)=F(-v)$ have the same induced connections, a reversed
geodesic is also auto-parallel and therefore they are (local) geodesics
as well. In particular, a forward geodesically complete Berwald space
is necessary backward geodesically complete. In such a case we just
say the Berwald space is geodesically complete. 

An important ingredient of the Berwald classification theorem was
the following result.
\begin{thm}
[Berwald metrization {\cite{Szabo1981}}] If $(M,F)$ is a Berwald
space with induced connection $\nabla$ then there exist (uniquely
defined) Riemannian metric $g$ whose Levi-Civita connection is $\nabla$.
In particular, $(M,F)$ and $(M,g)$ are affinely equivalent.
\end{thm}
Note there might be several Riemannian metrics compatible with $\nabla$.
However, the metric $g$ is intrinsically defined and if $(M',F')$
is isometric to $(M,F)$ then also $(M,g)$ and $(M',g')$.

Szab\'o actually showed only certain connections admit non-Riemannian
Berwald structures. His idea was to use the fact that an affine connection
induces a uniquely defined notion of parallel transport on the tangent
bundle. In particular, the Finsler norm $F_{x}$ at $T_{x}M$ needs
to be invariant by the holonomy group $\mathcal{H}_{x}$. However,
this is a rather rigid condition.
\begin{defn}
[Affine rigidity] A Berwald space $(M,F)$ is affinely rigid if for
every other Berwald space $(M,F')$ affinely equivalent to $(M,F)$,
i.e. their induced connection agree, it holds $F=\lambda F'$ for
some $\lambda>0$.
\end{defn}
Because every Berwald space is affinely equivalent to a Riemannian
manifold, the only affinely rigid Berwald spaces are Riemannian manifolds.
Now from the Berger-Simmons theorem and their corollaries we get the
following affine rigidity theorem.
\begin{thm}
[Affine rigidity classification]\label{thm:affine-rigid} Assume
$(M,F)$ is an irreducible geodesically complete Berwald space. Then
$(M,F)$ is either a affinely rigid Riemannian manifold or a higher
rank symmetric Berwald space. Furthermore, every continuous function
$G:TM\to\mathbb{R}$ which is invariant under parallel transport has
the following form
\[
G(v)=\varphi(F(v))
\]
for some $\varphi:[0,\infty)\to\mathbb{R}$. \end{thm}
\begin{proof}
Let $\mathcal{H}_{x}$ be the holonomy group of the connection acting
on $(T_{x}M,F_{x})$. If $\mathcal{H}_{x}$ acts transitively on the
unit sphere then any function $G$ is uniquely determined by the length
of the vector, i.e. $G(v)=\varphi(F(v))$. In particular, $F_{x}$
is a norm coming from an inner product, i.e. $(M,F)$ is Riemannian.
Now let $(M,F')$ be another Berwald space with the same connection.
Then $F'$ is preserved by parallel transport and therefore $F'=\varphi\circ F$.
By positive $1$-homogenity we see that $F'=\lambda F$, i.e. $(M,F)$
is affinely rigid. 

If the holonomy group does not act transitively then $(M,\nabla)$
is a symmetric space. If $(M,\nabla)$ is a higher rank symmetric
space then there is nothing to prove. So assume $(M,\nabla)$ has
rank $1$. Note that $M$ admits a Riemannian metric $g$ with Levi-Civita
connection $\nabla$.

Let $G$ be a continuous function invariant under parallel transport.
In particular, it is invariant under the holonomy group. Then Lemma
\ref{lem:holonRk1} implies that $G$ is uniquely determined by its
restriction to a one dimensional subspace $V$ of $T_{x}M$. Indeed,
for every $w\in T_{x}M$ there is an $h\in\mathcal{H}_{x}$ such that
$hw=v\in V$ and $G(w)=G(v)$. It remains to show that $G(v)=G(-v)$. 

Assume by contradiction $G(v)\ne G(-v)$. Then the following sets
are two disjoint open subsets covering the $S_{x}M$ w.r.t. the Riemannian
metric $g$: 
\begin{eqnarray*}
A_{<} & = & \{w\in S_{x}M\,|\, G(w)<\frac{G(v)+G(-v)}{2}\}\\
A_{>} & = & \{w\in S_{x}M\,|\, G(w)>\frac{G(v)+G(-v)}{2}\}.
\end{eqnarray*}
However, $S_{x}M$ is connected and cannot be covered by disjoint
open sets. Therefore, $G(v)=G(-v)$, i.e. $G(v)=\varphi(g(v))$. Similar
to the transitive case, $F$ must be induced by a Riemannian metric
and any other Berwald norm must be a multiple of $F$.
\end{proof}
In case the function $G(v)=F(v,v')$ where $(v,v')$ is a tangent
vector on $M=M_{1}\times M_{2}$ and $F$ a Finsler structure,  the
theorem has the following interpretation: Fix a vector $v'\in(TM_{2})_{y}$
with $F(0,v')<1$. Then 
\[
A_{v'}=\{v\in(TM_{1})_{x}\,|\, F(v,v')\le1\}
\]
is strictly convex in $TM_{(x,y)}$ with smooth boundary containing
the origin of $TM_{1}$ in its interior. Thus it represents a Finsler
norm on $TM_{1}$. If $(M_{1},F_{1})$ is affinely rigid this norm
must be the unique Finsler norm on $(M_{1},F_{1})$.
\begin{thm}
[Berwald de Rham decomposition]\label{thm:Berwald-deRham}Let $(M,F)$
be a simply connected, geodesically complete Berwald space. Then $(M,F)$
is smoothly isometric to a product $(M_{0}\times\cdots\times M_{n},\hat{F})$
given by the de Rham decomposition where $i:M\to M_{0}\times\cdots\times M_{n}$
is the de Rham isometry and $\hat{F}=i_{*}F$. Furthermore, $(M_{0},F_{0})$
is a Minkowski space and a curve in $(M,F)$ is a geodesic iff each
of its projection is a geodesic. 

If the de Rham decomposition has no factor which is a higher rank
symmetric space then the Finsler norm $\hat{F}$ is given by 
\[
\hat{F}(v_{0},\ldots,v_{n})=G(v_{0},F_{1}(v_{1}),\ldots,F_{n}(v_{n}))
\]
where $G:\mathbb{R}^{l+n}\to[0,\infty)$ is a Minkowski norm on $\mathbb{R}^{\dim M_{0}+n}$
which is symmetric in the last $n$-factors.\end{thm}
\begin{rem*}
This is a more explicit statement of Szab\'o Berwald classification
theorem \cite[Theorem 1.3]{Szabo2006}. Note, however, that the product
structure is uniquely determined by a holonomy invariant Minkowski
norm on a single tangent space $T_{\mathbf{x}}M=\oplus_{i}T_{x_{i}}M_{i}$. \end{rem*}
\begin{proof}
Let $\nabla$ denote the induced connection of $(M,F)$. Then $\nabla$
splits into irreducible components
\[
\nabla=\sum_{j=0}^{n}\nabla_{j}
\]
where $\nabla_{0}$ is the sum of one dimensional connections and
for $j\ge1$ the connections $\nabla_{j}$ are irreducible. Let $(M,g)$
be the affinely equivalent Riemannian manifold associated to $(M,F)$.
We apply the de Rham Decomposition Theorem to $(M,g)$, i.e. $(M,g)$
is isometric to a product $(M_{0}\times\cdots\times M_{n})$ such
that $M_{0}$ is the maximal Euclidean factor and each $(M_{i},g_{i})$
is irreducible with connection $(\pi_{i})_{*}\nabla_{i}$, for simplicity
also denoted by $\nabla_{i}$. Now each factor represents a family
of simply connected, totally geodesic submanifolds which are isometric
to $(M_{i},g_{i})$. As $(M,F)$ and $(M,g)$ are affinely equivalent
and $(M,g)$ isometric to $(M_{0}\times\cdots\times M_{n},\sum g_{i})$
we can push forward the Finsler norm to the product. In particular,
we can equip each $M_{i}$ with a Finsler norm such that $(M_{0},F_{0})$
is a Minkowski space and $(M_{i},F_{i})$ an irreducible Berwald space
whose connections are $\nabla_{i}$. As auto-parallel curves of $(M,F)$
and $(M,g)$ agree, we see that geodesics are given by geodesic of
the factors.

It remains show that the Finsler norm has the given form if none of
the factors is a higher rank symmetric space. Note that in this case
each $(M_{i},F_{i})$ is affinely rigid and each $F_{i}$ is induced
by a Riemannian metric. 

Now for fixed $v\in V_{0}$ the functions
\[
(w_{1},\ldots,w_{n})\mapsto F(v,w_{1},\ldots,w_{n})
\]
are invariant under parallel transport in each coordinate. Now using
Theorem\ref{thm:affine-rigid} we obtain by induction a function $G:\mathbb{R}^{k+n}$
such that 
\[
F(v,w_{1},\ldots,w_{n})=G(v,F_{1}(w_{1}),\ldots,F_{n}(w_{n})).
\]
Furthermore, we see that $F$ is a Minkowski norm iff $G$ is. 

Note that the map $i_{F}:(M,F)\to(M_{0}\times\cdots\times M_{n},F)$
defined via de Rham decomposition and reassigning the Finsler norm
is obviously an isometry. In particular, it is a diffeomorphism (use
the fact that affine equivalence is a diffeomorphism and then $i_{F}$
is the de Rham isometry combined with two affine equivalences, another
way is to use the results of \cite{Deng2002}). 
\end{proof}
It is not difficult to see that in case $M$ does not contain a higher
rank symmetric factor, the distance on $M$ induced by the Finsler
structure that has the following form
\[
d_{M}((x_{0},\ldots,x_{n}),(y_{0},\ldots,y_{n}))=G(x_{1}-x_{0},d_{M_{1}}(x_{1},y_{1}),\ldots,d_{M_{n}}(x_{n},y_{n})).
\]
In this case we say that the metric space $(M,d_{M})$ is a metric
product of a Minkowski space and $n$ metric space $(M_{i},d_{M_{i}})$.

For Berwald spaces containing a higher rank symmetric factor this
is in general not true because there are several functions $v_{1}\mapsto G(v_{1})$
invariant under the holonomy group which are not given as $\varphi(F_{1}(v_{1}))$.
In particular, in general 
\[
F(v_{1},v_{2})\ne G(F_{1}(v_{1}),v_{2}).
\]
Note that if $F$ is $C^{2}$ away from the zero section then equality
can only hold iff $F_{1}$ is $C^{2}$ at $v_{1}=0$. But then$F_{1}$
is necessarily Riemannian.

\section{Flag and Ricci Curvature}

Given an affine connection $\nabla$ one can define the (Riemann)
curvature tensor as 
\[
R(V,W)Z=\nabla_{V}\nabla_{W}Z-\nabla_{W}\nabla_{V}Z-\nabla_{[V,W]}Z
\]
where $V,W,Z$ are vector fields. If $\nabla$ is the Levi-Civita
connection of $(M,g)$ then the sectional curvature spanned by $V$
and $W$ is defined as 
\[
\mathcal{K}(V,W)=\frac{g(R(V,W)W,V)}{g(V,V)g(W,W)-g(V,W)^{2}}
\]
and the Ricci curvature 
\[
\operatorname{Ric}(V)=\sum_{i=1}^{n-1}\mathcal{K}(V,e_{i})
\]
were $\{e_{i}\}_{i=1}^{n}$ is an orthonormal basis at $T_{x}M$ with
$e_{n}=V/F(V)$. Equivalently, one can write the Ricci curvature of
a vector in tensor form as follows 
\[
\operatorname{Ric}(V^{i}\partial_{i})=R_{jik}^{i}V^{k}V^{k}
\]
where $R_{jlk}^{i}$ denotes the curvature tensor in local coordinates.
Note that this does not involve the metric tensor $(g_{ij})$. In
particular, it only depends on the connection $\nabla$.

Instead of now defining the curvatures directly from the Finsler structure
we use an interpretation of the flag and Ricci curvature by Shen (see
\cite{Shen2001,Shen1997}). Assume the integral curves of the vector
field $V$ are non-constant constant speed geodesics. In this case
we say $V$ is locally a geodesic field. Then flag curvature $\mathcal{K}^{V}(V,W)$
spanned by $V$ and $W$ with flag pole $V$ is just the sectional
curvature of $g(V)$ spanned by $V$ and $W$. In a similar way one
obtains the Ricci curvature $\operatorname{Ric}(V)$ as the trace
of the curvature tensor of $g(V)$. One can actually show that those
quantities are defined on the tangent space, i.e. given other geodesic
fields $V^{'}$ and $W^{'}$ with $V_{x}=V_{x}^{'}$ and $W_{x}=W_{x}^{'}$
then $\mathcal{K}^{V}(V,W)(x)=\mathcal{K}^{V^{'}}(V^{'},W^{'})(x)$
and $\operatorname{Ric}(V)(x)=\operatorname{Ric}(V^{'})(x)$.

A Finsler manifold $(M,F)$ has flag curvature bounded from below
by $K$ if 
\[
\mathcal{K}^{V}(V,W)\ge K
\]
 for all unit vector fields $V$ and $W$. Similarly, the Ricci curvature
is bounded from below by $K$ if 
\[
\operatorname{Ric}(v)\ge KF(v)^{2}
\]
for unit vectors fields $V$. It is easy to see that a lower bound
$K$ on the flag curvature implies the same bound on the Ricci curvature. 

A totally geodesic submanifold has the same (possibly better) flag
curvature bound as the containing manifolds. For the Ricci curvature
this might not be true in general. However, we have the following
simplification (see \cite[Prop. 6.6.2]{Shen2001}). 
\begin{lem}
Let $(M,F)$ be a Berwald space with induced connection $\nabla$
then 
\[
\mathcal{K}^{V}(V,W)=\frac{g_{V}(R(V,W)W,V)}{g_{V}(V,V)g(W,W)-g_{V}(V,W)^{2}},
\]
where $R$ is the curvature tensor of $\nabla$ and $g_{V}=g(V)$.\end{lem}
\begin{thm}
[Ricci affine invariance]\label{thm:Ricci}The Ricci curvature of
a Berwald space is an affine invariant and only depends on the induced
connection, i.e. if $(M,F)$ and $(M,F')$ are affinely equivalent
then 
\[
\operatorname{Ric}_{F}(v)=\operatorname{Ric}_{F'}(v).
\]
In particular, having non-positive, resp. non-negative Ricci curvature
is an affine invariant.\end{thm}
\begin{rem*}
This was implicitly used in \cite[proof of Lemma 4]{Deng2015} without
explicitly stating the affine invariance.\end{rem*}
\begin{proof}
Just note that the Ricci curvature is a non-metric contraction of
the curvature tensor. In particular, for $v=v^{i}\partial_{i}$ 
\[
\operatorname{Ric}_{g_{V}}(v)=R_{jik}^{i}(g_{v})v^{j}v^{k}.
\]
However, from the previous lemma, 
\[
R_{jik}^{i}(g_{v})v^{j}v^{k}=R_{jik}^{i}v^{j}v^{k}.
\]
\end{proof}
\begin{rem*}
Affine invariance of the Ricci curvature holds more general, i.e.
if two Finsler structure have the same Chern connection then their
Ricci curvatures agree. Indeed, if one denotes by $R_{jlk}^{i}$ the
curvature coefficients of the Chern connection then the arguments
above can be used without any change.
\end{rem*}
With the help of this we can easily show that higher rank symmetric
spaces are either compact and have positive Ricci curvature, or they
are non-compact and have negative Ricci curvature. Indeed, any symmetric
space $(M,g)$ admits a Riemannian metric, namely the Killing form,
such that $(M,g)$ is a Einstein space, i.e. $\operatorname{Ric}=\lambda g$.
If $(M,g)$ is not flat then $\lambda\ne0$. In case, $\lambda>0$
Myers' theorem implies $(M,g)$ is compact. Non-compactness in case
$\lambda<0$ follows by showing that there cannot be closed geodesics.
\begin{lem}
\label{lem:HigherRankRicci}A higher rank symmetric Berwald space
of non-negative Ricci curvature has strictly positive Ricci curvature
and is compact.
\end{lem}
Finally, some small applications to rigidity of the Berwald condition.
\begin{lem}
Assume $(M,F)$ is a Berwald spaces whose connection splits into two
(possibly reducible) factors $M_{1}$ and $M_{2}$ then $M$ cannot
have strictly positive, resp. strictly negative flag curvature at
a given flag pole, i.e. $\mathcal{K}^{V}(V,W)=0$ for all $V\ne0$
and some non-zero $W\ne V$.\end{lem}
\begin{proof}
Just note that the curvature tensor $R$ acts trivial on mixed vectors,
i.e.
\[
R(V_{1},V_{2})=0
\]
if $V_{1}$ is tangent to the $M_{1}$-factor and $V_{2}$ tangent
to the $M_{2}$-factor. Therefore, 
\[
\mathcal{K}^{V_{1}}(V_{1},V_{2})=\frac{g_{v_{1}}(R(V_{1},V_{2})V_{2},V_{1})}{g_{v_{1}}(V_{1},V_{1})g_{v_{1}}(V_{2},V_{2})-g_{V_{1}}(V_{1},V_{2})^{2}}=0.
\]
\end{proof}
\begin{lem}
If $(M,F)$ is an irreducible higher rank symmetric Berwald space
then it cannot have strictly its flag curvature strictly positive,
resp. strictly negative flag curvature at a given flag pole.\end{lem}
\begin{proof}
Through every point there is an at least two dimensional totally geodesic
flat manifolds. The flag curvature to this submanifold is zero, hence
prohibiting any strict bounds.
\end{proof}
Combing these facts we see that any strictly curved space must be
irreducible and cannot be a higher rank symmetric space. By Theorem
\ref{thm:affine-rigid} it must be Riemannian.
\begin{cor}
Let $(M,F)$ be a Berwald space. If its flag curvature (see below)
is strictly positive, resp. negative then it is irreducible and Riemannian.
\end{cor}
This can now be used to show that there are no strictly curved symmetric
Finsler spaces. This simple argument was already used by Matveev \cite{Matveev2015}.
\begin{cor}
Let $(M,F)$ be an irreducible (locally) symmetric Finsler space with
strictly positive (resp. strictly negative) flag curvature. Then $(M,F)$
is Riemannian. \end{cor}
\begin{proof}
Just note that every locally symmetric Finsler space is locally Berwald
(see \cite[Theorem 9.2]{Matveev2012}). Combining with the above we
obtain that the space is Riemannian.
\end{proof}

\section{Splitting theorems}

The splitting theorem is now an easy consequence of the Berwald de
Rham decomposition (Theorem \ref{thm:Berwald-deRham}) and affine
invariance of the Ricci curvature (see Theorem \ref{thm:Ricci}).
First recall the Cheeger-Gromoll splitting theorem.
\begin{lem}
[Cheeger-Gromoll Splitting Theorem] Let $(M,g)$ be a geodesically
complete Riemannian manifold of non-negative Ricci curvature containing
an isometrically embedded line $\eta:\mathbb{R}\to M$. Then $(M,g)$
is isometric to a Cartesian product $(M'\times\mathbb{R},g'+(\cdot,\cdot))$
where $g'$ induces a Riemannian metric on $M'$ whose Ricci curvature
is non-negative. In particular, the Levi-Civita connection of $M$
contains a flat one dimensional factor acting on the $\mathbb{R}$-factor. \end{lem}
\begin{thm}
\label{thm:CheegerGromoll}Let $(M,g)$ be a simply connected geodesically
complete Riemannian with non-negative Ricci curvature. Then $M$ is
isometric to a product $(M_{0}\times\cdots\times M_{n},\sum g_{i})$
where $M_{0}$ is a flat Euclidean space and each $(M_{i},g_{i})$
is a simply connected irreducible Riemannian manifold of non-negative.
Furthermore, each line contained in $M$ is tangent to the $M_{0}$-factor
and has constant $M_{i}$-coordinate.\end{thm}
\begin{proof}
This follows by combining the de Rham decomposition and the Cheeger-Gromoll
splitting Theorem above. Indeed, each irreducible factor needs to
have non-negative Ricci curvature. If there is a line not constant
on a factor $M_{i}$ then there is actually a line only tangent to
the $M_{i}$-factor. But then the connection on $M_{i}$ is reducible
which is impossible. 
\end{proof}
Note that if $M$ is not simply connected and the fundamental group
is torsion-free then each element in the fundamental group is represented
by a closed geodesic. But this geodesic lifts to a line in the universal
cover which has non-negative Ricci curvature as well. 
\begin{cor}
A geodesically complete Riemannian manifold $(M,F)$ with torsion-free
fundamental group and non-negative Ricci curvature splits into a product
as above where $M_{0}$ can be of the form $\mathbb{R}^{k}\times\mathbb{T}^{l}$.
Furthermore, every non-flat factor is simply connected.
\end{cor}
Now the Berwald splitting follows immediately.
\begin{thm}
\label{thm:Berwald-structure}Let $(M,F)$ be a simply connected geodesically
complete Berwald space of non-negative Ricci curvature. Then $M$
is a product $M=M_{0}\times\cdots\times M_{n}$ where $M_{0}=\mathbb{R}^{k}$
is flat and each $(M_{i},F_{i})$ is simply connected and either an
irreducible Riemannian manifold of non-negative Ricci curvature or
a higher rank symmetric Berwald space of compact type. 
\end{thm}
Again, if none of the factors is a higher rank symmetric space then
one obtains a more explicit form of the Finsler metric. Furthermore,
in the torsion-free case one can actually split of all non-flat factors.

Since every symmetric space is Einstein one might ask whether a similar
statement holds for higher rank symmetric Berwald spaces. More general,
what is the general structure of Berwald-Einstein spaces. 
\begin{defn}
[Berwald-Einstein] A Berwald space $(M,F)$ is called Berwald-Einstein
space if for every $v_{x}\in TM$ 
\[
\operatorname{Ric}(v_{x})=\lambda(x)F(v_{x})^{2},
\]
where $\lambda$ is a (fixed) function on $M$. 

We say that $(M,F)$ is Ricci-flat, if $\operatorname{Ric}=0$. Note
that any Ricci-flat Berwald space is Berwald-Einstein. Furthermore,
it is known that if $(M,F)$ is Riemannian then $\lambda$ needs to
be constant.\end{defn}
\begin{lem}
[Berwald-Einstein Rigidity {\cite{Deng2015}}] A connected Berwald-Einstein
space is either Ricci-flat or a Riemannian manifold. In each case
$\lambda$ is constant.\end{lem}
\begin{proof}
Let $U=\{x\in M\,|\,\lambda(x)\ne0\}$. Note that $\operatorname{Ric}$
is a quadratic form. Suppose $U\ne\varnothing$ then $F\big|_{U}$
is a quadratic form as well. In particular, it is induced by an inner
product. But as $F$ is preserved by parallel transport also $F\big|_{M\backslash U}$
must be induced by an inner product. In particular, $(M,F)$ is Riemannian. \end{proof}
\begin{thm}
[Berwald-Einstein Structure Theorem] A geodesically complete Berwald-Einstein
space $(M,F)$ is either an Einstein (Riemannian) manifold or it is
Ricci-flat and each non-flat factor $M_{i}$ is an irreducible Ricci-flat
Riemannian manifold and $F$ is given as in Theorem \ref{thm:Berwald-deRham}.\end{thm}
\begin{proof}
If $(M,F)$ is not Ricci-flat then it is Riemannian. In the Ricci-flat
case, note that a higher rank symmetric Berwald space has non-zero
Ricci curvature as there is an affinely equivalent non-flat Einstein
metric on such a space. Therefore, the full structure theorem applies.
\end{proof}
Because non-negative flag curvature implies non-negative Ricci curvature
we can apply the soul theorem to each of the factors and obtain a
weaker version of this theorem for Berwald spaces.
\begin{prop}
[Berwald soul theorem] Assume $(M,F)$ is a geodesically complete
Berwald space of non-negative flag curvature. Then there is a compact
totally convex, totally geodesic submanifold $S$, called soul of
$M$, such that $M$ is diffeomorphic to the normal bundle of $S$. \end{prop}
\begin{rem*}
We only show the existence of the soul. It is unclear whether the
Sharafutdinov map \cite{Sharafutdinov1978} is also distance non-increasing
in this setting. The contractive behavior of that map is proved via
a gradient flows of convex functions. In the Finsler setting, even
on flat Minkowski spaces, the gradient flow of convex function is
not necessarily contractive (see \cite[Theorem 3.2]{OS2012}).\end{rem*}
\begin{proof}
The proof follows by combining the splitting theorem above and the
soul theorem applied to each irreducible Riemannian factor of the
splitting. 

Namely, non-negative flag curvature implies that $M$ has non-negative
Ricci curvature. By the previous theorem $M=M_{0}\times\cdots\times M_{n}$
where $(M_{i},F_{i})$ are irreducible Riemannian or higher rank symmetric
spaces. Because $(M,F)$ has non-negative flag curvature, so does
every Riemannian factor $(M_{i},F_{i})=(M_{i},g_{i})$. For every
higher rank symmetric factor $(M_{i},F_{i})$, note that every affinely
equivalent Einstein metric $g_{i}$ on $M_{i}$ has non-negative sectional
curvature.

Take the Cartesian product $(M,\sum g_{i})$ of the Riemannian metrics
(with $g_{0}$ being any inner product on $M_{0}$). Then there is
a compact totally convex, totally geodesic Riemannian submanifold
$S$ of $(M,\sum g_{i})$. Because $(M,F)$ and $(M,\sum g_{i})$
are affinely equivalent, $S$ is also a compact totally convex, totally
geodesic submanifold of $(M,F)$. Furthermore, the soul theorem shows
that $M$ is diffeomorphic to the normal bundle of $S$. If one uses
the forward exponential map on $(M,F)$ instead of the exponential
map of $(M,g)$ then this diffeomorphism is on $C^{2}$ and $C^{\infty}$
iff $(M,F)$ is actually a Riemannian manifold and $F$ induced by
$g$.
\end{proof}
As an application we can show a conjecture by Lakzian \cite{Lakzian2014}
on non-negatively curved reversible Berwald spaces.
\begin{cor}
A geodesically complete Berwald space $(M,F)$ of non-negative flag
curvature with large volume growth is diffeomorphic to $\mathbb{R}^{n}$.\end{cor}
\begin{proof}
Using the previous preposition we see that $(M,F)$ is diffeomorphic
to the normal bundle of a compact totally convex, totally geodesics
submanifold $S$. If $S$ is not a point then it contains a closed
geodesic, and then also $(M,F)$. Indeed, $(S,F_{S})$ is affinely
equivalent to a geodesically complete, compact Riemannian manifold
$(S,g_{S})$. Because $(S,g_{S})$ is Riemannian it contains a closed
geodesics (see \cite{Klingenberg1995}). 

However, Lakzian showed \cite{Lakzian2014} that a non-negatively
curved Berwald space with large volume growth cannot contain any closed
geodesics. Therefore, $S$ is a point and $M$ diffeomorphic to $\mathbb{R}^{n}$. 

Note that the assumption on reversibility is not needed. Indeed, if
$(M,F,\mathbf{m})$ has large volume growth (see \cite[3.2]{Lakzian2014})
then so has $(M,g,\mathbf{m})$ where $(M,g)$ is affinely equivalent
to $(M,F)$ and has non-negative sectional curvature.\end{proof}
\begin{rem*}
Lakzian actually showed the non-existence of closed geodesics for
Berwald spaces of non-negative radial flag curvature with large volume
growth. If non-negative radial flag curvature of $(M,F)$ implies
non-negative sectional curvature of an affinely equivalent Riemannian
manifold $(M,g)$ then the main result of \cite{Lakzian2014} can
be reduced to the Riemannian setting as well. It is likely that this
follows if one shows that every higher rank symmetric Berwald space
has the same curvature bound as its Riemannian equivalent.
\end{rem*}

\section{Weighted Ricci curvature}

Similar to Riemannian manifolds there is a Finsler version of \emph{weighted
Ricci curvature}, see \cite{Ohta2009}. However, instead of the general
version we will use a Berwald version resembling the weighted Ricci
curvature for Riemannian manifolds. With this we can apply again the
structure theorem for Riemannian manifold with non-negative weighted
Ricci curvature.

Let us first define the weighted Ricci curvature: Denote by $\operatorname{vol}_{F}$
the \emph{Busemann-Hausdorff measure}. One can show that this measure
is a multiple of the volume form $\operatorname{vol}_{g}$ of an affinely
equivalent Riemannian structure $g$ on $M$ (follows from the proof
of \cite[Prop. 7.3.1]{Shen2001}). 

Now we can equip any Berwald space with a smooth measure $\mathbf{m}$
which is absolutely continuous w.r.t. $\operatorname{vol}_{F}$. Then
there is a function $\Psi:M\to\mathbb{R}$ such that 
\[
\mathbf{m}=e^{-\Psi}\operatorname{vol}_{F}.
\]

\begin{defn}
[Weighted Ricci curvature] Denote by $n$ the dimension of $M$ and
let $N'\in(n,\infty)$ and $v=\dot{\eta}(0)$ be a tangent vector.
We define the following quantities:
\begin{itemize}
\item $\operatorname{Ric}_{n}(v)=\begin{cases}
\operatorname{Ric}(v)+\left(\Psi\circ\eta\right)''(0) & \mbox{if \ensuremath{\left(\Psi\circ\eta\right)}'(0)=0}\\
-\infty & \mbox{if \ensuremath{\left(\Psi\circ\eta\right)}'(0)\ensuremath{\ne}0}
\end{cases}$
\item $\operatorname{Ric}_{N'}(v)=\operatorname{Ric}(v)+\left(\Psi\circ\eta\right)''(0)+\frac{\left(\left(\Psi\circ\eta\right)'(0)\right)^{2}}{N'-n}$
\item $\operatorname{Ric}_{\infty}(v)=\operatorname{Ric}(v)+\left(\Psi\circ\eta\right)''(0).$
\end{itemize}
Note that we have $\operatorname{Ric}_{N}(cv)=c^{2}\operatorname{Ric}_{N'}$
for $N\in[n,\infty]$ and $c>0$. 
\end{defn}
Now we say that $(M,F,\mathbf{m})$ the $N$-dimensional Ricci curvature
bounded from below if 
\[
\operatorname{Ric}_{N}(v)\ge KF(v)^{2}.
\]

\begin{rem*}
If $F$ is induced by a Riemannian metric then it is exactly the definition
of weighted (Riemannian) Ricci curvature. Furthermore, one can show
that it is equivalent to Ohta's original Finsler definition, see \cite{Ohta2009,Ohta2013a}.
Indeed, first note that there is a function $\Phi:M\to\mathbb{R}$
such that 
\[
\operatorname{vol}=e^{-\Phi}\operatorname{vol}_{\dot{\eta}}
\]
where $\operatorname{vol}_{\dot{\eta}}$ denote the Riemannian volume
measure w.r.t. $g(\dot{\eta})$. As $\operatorname{vol}_{F}$ is the
Busemann-Hausdorff measure $(\Phi\circ\eta)'(0)\equiv0$ \cite[Prop. 7.3.1]{Shen2001}.
Now if 
\[
\mathbf{m}=e^{-\tilde{\Psi}(\eta)}\operatorname{vol}_{\dot{\eta}}
\]
for some function $\tilde{\Psi}:M\to\mathbb{R}$ then $\Psi=\tilde{\Psi}-\Phi$.
But then $(\tilde{\Psi}\circ\eta)'(0)=(\Psi\circ\eta)'(0)$ and $(\tilde{\Psi}\circ\eta)''(0)=(\Psi\circ\eta)''(0)$.
In particular, the definitions of weighted Ricci curvatures agree.\end{rem*}
\begin{prop}
Weighted Ricci curvature is an affine invariant.\end{prop}
\begin{proof}
Let $(M,F)$ and $(M,F')$ be affinely equivalent. Equip $M$ with
a measure $\mathbf{m}$. Then we have the decomposition 
\[
\mathbf{m}=e^{-\Phi_{F}}\operatorname{vol}_{F}=e^{-\Phi_{F'}}\operatorname{vol}_{F'}.
\]
As the Busemann-Hausdorff measures are multiples of each other we
must have $\Phi_{F}=\Phi_{F'}+c$ for some constant $c\in\mathbb{R}$.
But geodesics as curves agree so that we must have $\operatorname{Ric}_{N}^{F}=\operatorname{Ric}_{N}^{F'}$.
\end{proof}
Similar to the discussion above we get the following the structure
theorem for non-negative weighted Ricci curvature: By an isometrically
embedded line $\eta:\mathbb{R}\to M$ we implicitly equip $\mathbb{R}$
either with a symmetric metric or an asymmetric metric to allow for
lines inside the Minkowski factor.
\begin{lem}
[Weighted Berwald Splitting Theorem] Let $(M,F,\mathbf{m})$ be a
geodesically complete Berwald space with non-negative $N$-Ricci curvature.
If $M$ contains an isometrically embedded line then there is a measure
preserving isometry onto $(\mathbb{R}\times M',F',\lambda\times\mathbf{m}')$. \end{lem}
\begin{thm}
Let $(M,F,\mathbf{m})$ be a geodesically complete Berwald space with
non-negative $N$-Ricci curvature. Then there is a measure-preserving
isometry onto $(M_{0}\times M',\hat{F},\lambda^{k}\mathbf{m}')$ where
$M_{0}=\mathbb{R}^{k}$ is flat, $\lambda^{k}$ is the Lebesgue measure
on $\mathbb{R}^{k}$ and $(M',\hat{F}_{M'},\mathbf{m}')$ has non-negative
$(N-k)$-dimensional Ricci curvature. In particular, all lines are
entirely tangent to the flat factor. Furthermore, none of the factors
of $M'$ is a higher rank symmetric space of non-compact type.\end{thm}
\begin{rem*}
In contrast to the general structure theorem it is not possible to
split of all flat factors as it is not clear that all but the flat
factors of $M'$ are simply connected.
\end{rem*}
Similar to the structure theorem the Finsler norm can be given more
explicitly if none of the factors of $M'$ is a higher rank symmetric
Berwald space of compact-type.

\section{Conclusion and Questions}

We have shown that excluding symmetric factors, the topology and many
geometric properties reduce to the study of metric products of irreducible
Riemannian manifolds. Furthermore, (weighted) Ricci curvature is an
affine invariant in the category of Berwald spaces. In particular,
non-negative/non-positive Ricci curvature is an affine invariant.
Since non-negative $N$-Ricci curvature is equivalent to Lott-Sturm-Villani's
curvature dimension condition $CD(0,N)$ (see \cite{Ohta2009} for
definition and proof for Finsler manifolds), one might ask whether
this holds more generally.
\begin{question}
Assume $(X_{i},d_{i},\mathbf{m}_{i})$ are metric spaces satisfying
the $CD(0,N_{i})$-condition. Is it true that the metric product $(X_{1}\times\cdots\times X_{n},d,\times_{i}\mathbf{m}_{i})$
with 
\[
d(\mathbf{x},\mathbf{y})=F(d(x_{1},y_{1}),\ldots,d(x_{n},y_{n}))
\]
satisfies the $CD(0,\sum N_{i})$-condition?
\end{question}
Easier to answer might be the question whether this holds for the
metric products of $RCD$-spaces. More general one could ask whether
$CD(K_{i},N_{i})$ implies $CD(K,N)$ for some $K$.

The concept of affine equivalence can be defined also for geodesic
metric spaces. Namely, two metrics on $X$ are affinely equivalent
if every geodesic of $(X,d_{1})$ is a geodesic of $(X,d_{2})$.
\begin{question}
Assume $(X,d)$ is an $RCD(0,N)$-space. Does every affinely equivalent
metric space $(X,d')$ satisfy $CD(0,N)$?
\end{question}
It is not difficult to see that this is true in the compact setting
if the definition of $CD(0,N)$ is relaxed to allow convexity of the
Renyi entropy to hold along ``some'' curve in the Wasserstein space
$\mathcal{P}_{2}(X,d)$, compare \cite[9.2]{AmbGigSav2008}.

\bibliographystyle{amsalpha}
\bibliography{bib}

\end{document}